\documentclass[11pt,letterpaper]{amsart}
\usepackage{amsthm}

\usepackage{mathrsfs}
\usepackage{mathtools}
\usepackage{slashed}
\usepackage{tikz-cd}

\usepackage{amssymb}
\usepackage{amsfonts}
\usepackage{amsmath}
\usepackage{graphicx}
\usepackage{xcolor}
\usepackage{amsmath}
\usepackage{mathrsfs}
\usepackage{stmaryrd}
\usepackage{epsfig,color}
\usepackage{blindtext}
\usepackage{enumerate}
\usepackage[colorlinks=true,linkcolor=red,citecolor=blue]{hyperref}
\usepackage[noabbrev,capitalize]{cleveref}
\usepackage[msc-links, lite, alphabetic]{amsrefs}

\usepackage{url}
\usepackage{nicefrac,mathtools}
\DeclareGraphicsExtensions{.pdf,.jpeg,.png}
\usepackage{epstopdf}
\usepackage{cancel} 
\usepackage[normalem]{ulem} 
\usepackage{verbatim} 
\usepackage{enumitem} 
\usepackage{color}

\usepackage{geometry}
\DeclareMathOperator{\divsymb}{div}

\DeclareMathOperator{\vol}{Vol}

\DeclareMathOperator{\Ric}{Ric}

\newcommand{\ha}{\widehat{a}}

\newcommand{\mA}{\mathcal{A}}
\newcommand{\mB}{\mathcal{B}}

\newcommand{\mb}{\mathbb}

\newcommand{\R}{\mathbb{R}}

\newcommand{\defeq}{\mathrel{\mathop:}=}

\def\XXint#1#2#3{{\setbox0=\hbox{$#1{#2#3}{\int}$ }
\vcenter{\hbox{$#2#3$ }}\kern-.6\wd0}}

\def\sideremark#1{\ifvmode\leavevmode\fi\vadjust{\vbox to0pt{\vss
 \hbox to 0pt{\hskip\hsize\hskip1em
 \vbox{\hsize3cm\tiny\raggedright\pretolerance10000
 \noindent #1\hfill}\hss}\vbox to8pt{\vfil}\vss}}}

\newtheorem{theorem}{Theorem}[section]
\newtheorem{mainthm}{Theorem}

\newtheorem{proposition}[theorem]{Proposition}
\newtheorem{lemma}[theorem]{Lemma}
\newtheorem{corollary}[theorem]{Corollary}

\theoremstyle{definition}

\newtheorem{question}[theorem]{Question}

\newtheorem{remark}[theorem]{Remark}

\numberwithin{equation}{section}

\allowdisplaybreaks[4]

\title{Topology of 3-manifolds with nonnegative scalar curvature and positive harmonic functions}

\author[Z. Yan]{Zetian Yan}
\address[Z. Yan]{Department of Mathematics \\ UC Santa Barbara \\ Santa Barbara \\ CA 93106 \\ USA}
\email{ztyan@ucsb.edu}

\author[X. Zhu]{Xingyu Zhu}
\address[X. Zhu]{Department of Mathematics \\ Michigan State University \\ East Lansing \\ MI 48824 \\ USA}
\email{zhuxing3@msu.edu}


\begin{document}
\begin{abstract}
We study complete \(3\)-manifolds with nonnegative scalar curvature under additional regularity assumptions. We prove that a contractible such manifold is diffeomorphic to $\mathbb{R}^3$, and that an open handlebody admitting such a metric must have genus at most \(1\). The proof uses exhaustions by level sets of harmonic functions and refined average gradient estimates.
\end{abstract}

\maketitle
\section{Introduction}
A central theme in Riemannian geometry is to understand the topological restrictions imposed by curvature bounds. For scalar curvature, the foundational works of Schoen--Yau \cite{SchoenYau79} and Gromov--Lawson \cite{GromovLawson} provide the basic framework. Nevertheless, even in dimension \(3\), determining the topology of manifolds that admit metrics of positive scalar curvature remains a longstanding open problem.

In this paper, we study the topology of \(3\)-manifolds admitting nonnegative scalar curvature under additional assumptions. Our motivation comes from the following question of J. Wang:
\begin{question}[\cites{Wang2024-1,Wang2024-2}]
    If \(M\) is a contractible \(3\)-manifold admitting a complete metric of nonnegative scalar curvature, must \(M\) be diffeomorphic to \(\mathbb{R}^3\)?
\end{question}
We are also motivated by the following question of Gromov, as in Chodosh--Lai--Xu \cite{Chodosh-Lai-Xu2025}:
\begin{question}[\cite{fourlectures}*{\S 3.10.2}]
    If an open handlebody \(M_\gamma\) of genus \(\gamma\) admits a complete metric of nonnegative scalar curvature, must it be the case that \(\gamma \le 1\)?
\end{question}

When \(M\) or \(M_\gamma\) has bounded geometry, both questions were answered affirmatively in \cite{Chodosh-Lai-Xu2025}. In this work, we substantially weaken those strong curvature and noncollapsing assumptions, at the expense of imposing an analytic condition on the behavior of the Green's function at infinity in the nonparabolic case. We now introduce the regularity assumptions needed for our method.
\begin{itemize}
    \item The Ricci curvature is bounded from below, that is,
    \begin{equation}\label{eq:Ricci-lower-bound}
        \Ric_g \ge K, \qquad K \le 0.
    \end{equation}
    
    \item When \(M\) is nonparabolic, its minimal positive Green's function \(G\) with pole at \(p \in M\) vanishes at infinity, that is,
    \begin{equation}\label{eq:Green-vanish}
        \lim_{d_g(x,p)\to\infty} G(x)=0.
    \end{equation}
\end{itemize}

Under these additional assumptions, we obtain the same affirmative conclusions:

\begin{mainthm}\label{Thm:Diffeo-to-R3}
    Let \((M,g)\) be a contractible complete Riemannian \(3\)-manifold with nonnegative scalar curvature \(R_g \ge 0\). Suppose in addition that \((M,g)\) satisfies the regularity assumptions \eqref{eq:Ricci-lower-bound} and \eqref{eq:Green-vanish}. Then \(M\) is diffeomorphic to \(\mathbb{R}^3\).
\end{mainthm}

\begin{mainthm}\label{genus-at-most-1}
    Let \(M_\gamma\) be the interior of a handlebody of genus \(\gamma\). Suppose that \(M_\gamma\) admits a complete Riemannian metric \(g\) with nonnegative scalar curvature \(R_g \ge 0\), and that \((M_\gamma,g)\) satisfies the regularity assumptions \eqref{eq:Ricci-lower-bound} and \eqref{eq:Green-vanish}. Then \(\gamma \le 1\).
\end{mainthm}

To motivate our general strategy, we briefly review several related results. In the case of uniformly positive scalar curvature, J. Wang \cite{Wang2025} established a complete classification: manifolds admitting complete metrics of uniformly positive scalar curvature are connected sums of spherical space forms and \(\mathbb{S}^2 \times \mathbb{S}^1\). The key idea is to exhaust the manifold by domains whose boundaries are diffeomorphic to \(\mathbb{S}^2\).

In a related direction, J. Wang \cites{Wang2024-1,Wang2024-2} proved that if a contractible Riemannian \(3\)-manifold \((M,g)\) with \(R_g \ge 0\) admits an exhaustion by solid tori, then \(M \cong \mathbb{R}^3\). A common feature of both arguments is that the topology of the boundaries appearing in an exhaustion imposes strong topological restrictions on the ambient manifold, see Lemmas \ref{lem:sphere-exhaustion-R3}, \ref{lem:torus-exhaustion-R3}, and \ref{lem:handlebody-genus-boundary}. This viewpoint led Chodosh--Lai--Xu to study the exhaustion whose boundaries are \(\mathbb{S}^2\) or \(\mathbb{T}^2\), constructed via inverse mean curvature flow.

We pursue the same general strategy, but with a different technical tool. The domains used to construct the exhaustion no longer arise from inverse mean curvature flow, instead, they are given by level sets of certain positive proper harmonic functions. The underlying idea originates in the monotone quantities introduced in \cite{ColdingMonotonicity} for manifolds with nonnegative Ricci curvature.

More precisely, let \((M^n,g)\) be a nonparabolic \(n\)-manifold with \(\Ric_g \ge 0\), and let \(G(x)\) be the minimal Green's function with pole at \(p \in M\). Since in this case \(G(x)\) behaves like \(d_g(x,p)^{2-n}\) at all scales, the function
\[
b \defeq G^{\frac{1}{2-n}}
\]
may be regarded as a regularized distance function to \(p\). For \(\beta>0\), define
\[
A_\beta(r)=r^{1-n}\int_{\{b=r\}} |\nabla b|^{1+\beta}\,d\sigma.
\]
It was shown in \cite{ColdingMonotonicity} that \(A_\beta(r)\) is monotone in \(r\) whenever \(\Ric_g \ge 0\). In particular, \(A_0(r)=4\pi\) is a constant, and \(A_2(r)\) was used in \cite{CM14} to study the uniqueness of asymptotic cones of Ricci-flat manifolds.

For questions involving scalar curvature in dimension $n=3$, the relevant quantity is 
\[A_1(r)=r^{-2}\int_{\{b=r\}}|\nabla b|^2, \quad b=G^{-1}\] 
Colding--Minicozzi \cite{Colding-Minicozzi2025} related estimates for \(A_1(r)\) to nonnegative scalar curvature, refining an earlier monotonicity formula of Munteanu--Wang \cite{Munteanu--Wang2026}. We observe that, in the setting of \cite{Colding-Minicozzi2025}, if one assumes that the regular level sets \(\{b=r\}\) have positive genus for all sufficiently large \(r\), then the estimates for \(A_1\) can be strengthened further, leading to a contradiction. As a consequence, we obtain an exhaustion by level sets of \(b\) whose genus is zero.

Two issues remain. The first is the parabolic case, which must be treated separately. The second is more technical: one needs appropriate hypotheses under which the above argument can be carried out in the nonparabolic case. We begin with the latter issue.

A crucial prerequisite for the level-set methods in \cite{Munteanu--Wang2026} and \cite{Colding-Minicozzi2025} is the properness of the Green's function \(G\), equivalently, the condition that \(G(x)\to 0\) as \(x\to\infty\). Without properness, the positive level sets \(\{b=r\}\) may fail to be compact, and consequently \(A_\beta(r)\) may not even be well defined. The argument for the connectedness of \(\{b=r\}\) also relies on the compactness of these level sets, see \cite{MWcomparison}*{Proof of Lemma 2.3}.  In general, one only has
\[
\liminf_{x\to\infty} G(x)=0,
\]
which is insufficient for our purposes. 
We therefore impose the decay of \(G\) at infinity as assumption \eqref{eq:Green-vanish}.

\begin{remark}
There are several geometric conditions that guarantee the vanishing of the Green's function at infinity, namely \eqref{eq:Green-vanish}.
\begin{itemize}
    \item If \((M,g)\) has positive bottom of spectrum, that is, \(\lambda_1(M)>0\), has Ricci curvature lower bound \eqref{eq:Ricci-lower-bound} and is volume noncollapsed, i.e.
    \[
    \inf_{x\in M}\vol_g(B_1(x))>0,
    \]
    then \eqref{eq:Green-vanish} holds, see \cite{MWcomparison}*{Proof of Theorem 4.3}. However, $\lambda_1(M)>0$ is not necessary as $\lambda_1(\R^n)=0$ but its Green's function vanishes at infinity.
    
    \item If \((M,g)\) satisfies the lower Ricci curvature bound as in \eqref{eq:Ricci-lower-bound}, and the global Sobolev inequality with $\nu>2$,
    \[
    \left(\int_M f^{\frac{2\nu}{\nu-2}}\, dV\right)^{\frac{\nu-2}{2\nu}}
    \le C \left(\int_M |\nabla f|^2\, dV\right)^{\frac12},
    \qquad \forall f\in C_c^\infty(M).
    \]
    which is the case for minimal hypersurfaces in Euclidean \(\mathbb{R}^{n+1}\) for $\nu=n$, then \eqref{eq:Green-vanish} holds, see \cite{ChL}*{Proof of Property (5) in Proposition 2.1}. See also \cite{ChavelFeldman}*{Theorem 2}.   
\end{itemize}

\end{remark}

We now turn to the parabolic case. Our main new contribution is a treatment of the parabolic setting parallel to that of \cite{Colding-Minicozzi2025}. In \cite{YZ26}, the authors introduced a substitute for the quantity \(A_2(r)\) and proved uniqueness results for asymptotic limits in the Ricci-flat, linear volume growth setting, in parallel with \cite{CM14}. Since manifolds with nonnegative Ricci curvature and linear volume growth are parabolic, this provides the appropriate framework for the parabolic case. The monotone quantity considered in \cite{YZ26} is
\[
\mA_2(r)=\int_{\{v=r\}} |\nabla v|^3 \, d\sigma,
\]
where \(v\) is a proper harmonic function on the end that plays the role of a distance function.

From a geometric point of view, \(A_2(r)\) detects the cone structure in the blow-down limit, since \(A_2'(r)\) is essentially given by the \(L^2\)-norm of the traceless Hessian of \(b^2\), which in turn controls the Gromov--Hausdorff distance to a metric cone via the principle that volume cones imply metric cones \cite{CheegerColdingAlmostRigidity}. By contrast, \(\mA_2(r)\) detects splitting structure at infinity, since \(\mA_2'(r)\) is essentially given by the \(L^2\)-norm of the Hessian of \(u\), which controls the Gromov--Hausdorff distance to a splitting space \cite{CheegerColdingAlmostRigidity}.

From an analytic perspective, one is naturally led to expect that if \(\mA_1\) is defined in the parabolic case using the barrier function of a parabolic end in the same way that \(A_1\) is defined in the nonparabolic case, then \(\mA_1\) should serve as the appropriate analogue of \(A_1\). We define for the barrier function $u$
\[
\mA_1(r)=\int_{\{u=r\}} |\nabla u|^2 d\sigma
\]
We confirm this intuition and show that the same general argument outlined above for the nonparabolic case also applies in the parabolic setting. The main difference is that, in the parabolic case, we obtain an exhaustion whose boundaries are all diffeomorphic to \(\mathbb{T}^2\).

Finally, we refer the reader to \cite{Colding-Minicozzi2025}*{Appendix A}, and the references therein, for the differentiability of \(\mA_1\) and \(A_1\), as well as the continuity of the other level set integrals that arise in the argument.

The paper is organized as follows. In Section~2, we collect the topological preliminaries needed in the proof, including results on ends of contractible manifolds, connectedness of level sets of harmonic functions, and genus bounds for open handlebodies. In Section~3, we prove the main theorems by analyzing level set exhaustions arising from harmonic functions. We first treat the parabolic case, where the exhaustion is constructed from the barrier function on the end, and then turn to the nonparabolic case, where the exhaustion is given by the level sets of the minimal positive Green's function. In each case, the key step is to show that the relevant level sets have controlled topology, which then allows us to apply the topological lemmas from Section~2 to conclude the proofs of Theorems~A and~B.

\subsection*{Acknowledgements}
The authors thank Jian Wang, Guofang Wei and Nan Wu for their interest in this work. Z.Y. is supported by an AMS–Simons Travel Grant. X.Z. is supported by an AMS–Simons Travel Grant.



\section{Topological preliminaries}

In this section, we collect several topological preliminaries, which are drawn from \cites{Wang2024-1,Wang2024-2,Chodosh-Lai-Xu2025,MWcomparison,Munteanu--Wang2026}.

\begin{lemma}[\cite{Chodosh-Lai-Xu2025}*{Lemma~2.1}]\label{lem:one-end-contractible}
Assume $M$ is contractible and $\dim M \geq 2$. Then $M$ has only one end.
\end{lemma}

\begin{lemma}\label{lem:contractible-orientable}
If $M$ is contractible, then $M$ is orientable.
\end{lemma}

\begin{proof}
Since $M$ is contractible, it is homotopy equivalent to a point. Hence
$H^1(M;\mathbb Z_2)=0.$ Then the first Stiefel--Whitney class $w_1(M)=0$.
Therefore, $M$ is orientable.
\end{proof}


In Lemma \ref{lem:sphere-exhaustion-R3} and \ref{lem:torus-exhaustion-R3} below, we suppose that $M$ is a contractible $3$-manifold, and that
\[
\Omega_1 \Subset \Omega_2 \Subset \cdots \Subset M,
\qquad
\bigcup_{i=1}^{\infty}\Omega_i = M,
\]
is an exhaustion, with each $\partial\Omega_i$ smooth and connected.

\begin{lemma}[\cite{Husch-Price1970}, cf.\cite{Chodosh-Lai-Xu2025}*{Lemma~2.3}]\label{lem:sphere-exhaustion-R3}
If $\partial\Omega_i \cong\mb S^2$ for all $i = 1,2,\dots$, then $M \cong \mathbb{R}^3$.
\end{lemma}

The following result is drawn from \cite{Chodosh-Lai-Xu2025}*{Lemma 2.4} and is essentially due to J.~Wang \cite{Wang2024-1}*{Corollary~1.3}, see also \cite{Wang2024-2}.

\begin{lemma}[\cite{Chodosh-Lai-Xu2025}*{Lemma 2.4}]\label{lem:torus-exhaustion-R3}
If $\partial\Omega_i \cong \mb T^2$ for all $i = 1,2,\dots$, and $M$ admits a complete metric of nonnegative scalar curvature, then $M \cong \mathbb{R}^3$.
\end{lemma}

The following lemma concerns the connectedness of level sets of positive harmonic functions. We state it in a form adapted to our setting.

\begin{lemma}[\cite{Munteanu--Wang2026}*{Lemma 2.2}, \cite{MWcomparison}*{Lemma 2.3}]\label{lem:levelset-connected}
    Let \((M,g)\) be a complete manifold with one end and finite first Betti number. Then the following hold:
    \begin{itemize}
        \item If \(u\) is a proper harmonic function on the end \(E\), then every level set of \(u\) is connected.
        \item If \(G\) is the minimal Green's function harmonic function defined on \(M\setminus\{p\}\) and vanishes at infinity \eqref{eq:Green-vanish}, then there exists \(t_0>0\) such that the level sets \(\{G=t\}\) are connected for all \(t\le t_0\). If, in addition, \(M\) has vanishing first Betti number, then every positive level set of \(G\) is connected.
    \end{itemize}
\end{lemma}

\begin{remark}
A contractible manifold clearly has vanishing first Betti number, and by Lemma \ref{lem:one-end-contractible} it has only one end. An open handlebody of genus \(\gamma\) has first Betti number \(\gamma<\infty\), and its interior has only one end since its boundary is connected. Thus, Lemma \ref{lem:levelset-connected} applies in both settings of Theorems \ref{Thm:Diffeo-to-R3} and \ref{genus-at-most-1}.
\end{remark}

Finally, we record the lemma we will use to control the genus of a handlebody.

\begin{lemma}[\cite{Chodosh-Lai-Xu2025}*{Lemma~2.7}]\label{lem:handlebody-genus-boundary}
Let $M^3$ be an open handlebody. Let
\[
\Omega_1 \Subset \Omega_2 \Subset \cdots \Subset M
\]
be a $C^1$ exhaustion of $M$, such that each $\Omega_i$ and $M\setminus \Omega_i$ is connected. Then
\[
\operatorname{genus}(\partial\Omega_i)\geq \operatorname{genus}(M)
\]
for all sufficiently large $i$.
\end{lemma}


\section{Proof of the main theorems}
Recall that \((M,g)\) is either parabolic or nonparabolic. In the parabolic case, we construct an exhaustion using the level sets of the barrier function, namely a positive proper harmonic function on the parabolic end, extended by zero to the whole manifold. In Proposition \ref{prop:Parabolic-exhaustion}, we will show that the boundaries of the domains in this exhaustion are all diffeomorphic to either \(\mathbb{S}^2\) or \(\mathbb{T}^2\).

In the nonparabolic case, we construct an exhaustion using the level sets of the minimal positive Green's function. In Proposition \ref{prop:Nonparabolic-exhaustion}, we will show that the boundaries of the domains in this exhaustion are all diffeomorphic to \(\mathbb{S}^2\). It then follows from Lemmas \ref{lem:sphere-exhaustion-R3} and \ref{lem:torus-exhaustion-R3} that Theorem \ref{Thm:Diffeo-to-R3} holds. Similarly, Theorem \ref{genus-at-most-1} follows from Lemma \ref{lem:handlebody-genus-boundary}.

In the remainder of the paper, we focus on bounding the genus of these level sets.

\subsection{Parabolic case}
In this case, there exists a nontrivial proper harmonic function $u$ on the unique end of $M$ with zero boundary value by \cite{LiTamHarmonic}*{Lemma 1.2}, and $u$ can be chosen to be proper by \cite{Nakai1962}. Extending this function by zero yields a continuous function on all of $M$, and its level sets may be used to exhaust $M$. We show that this exhaustion has the desired boundary properties.

\begin{proposition}\label{prop:Parabolic-exhaustion}
    Let $(M,g)$ be a complete Riemannian $3$-manifold with one end, finite first Betti number and $R_g\ge 0$. If $(M,g)$ is parabolic and satisfies the regularity assumption \eqref{eq:Ricci-lower-bound}, then $M$ admits an exhaustion $\Omega_i$ such that for all $i$ either $\partial\Omega_i\cong \mb S^2$ or  $\partial\Omega_i\cong \mb T^2$.
\end{proposition}

We layout our setting as follows. Let $E$ be the (parabolic) end of $M$, there exists a nontrivial proper harmonic function $u$ on $E$ satisfying
\begin{align*}
\left\{\begin{array}{cc}
    &u=0 \text{ on } \partial E,\\
    &\Delta u=0 \text{ on }E,\\ 
    &u(x)\to \infty \text{ as } x\to \infty.
\end{array}\right.
\end{align*}
By Sard's theorem and the properness of $u$, almost every level set of $u$ is a smooth closed surface. We refer to such level sets as \emph{regular level sets}. By Lemma \ref{lem:levelset-connected} and \ref{lem:contractible-orientable}, each regular level set of $u$ is connected and oriented, see also \cite{Munteanu--Wang2026}*{Lemma~2.2}.

To obtain a contradiction in the proof of Proposition \ref{prop:Parabolic-exhaustion}, suppose that there exists $r_0>0$ such that, for every regular value $r\ge r_0$, the level set
\[
\Xi_r:=\{u=r\}
\]
is a smooth closed surface of genus at least one, namely
\begin{equation}
    \mathrm{genus}(\Xi_r)\ge 1.
\end{equation}
Otherwise, one could find an exhaustion whose boundaries are diffeomorphic to $\mathbb S^2$.

We start with some local computations. Define
\[
\mA_0(r)\defeq\int_{\Xi_r} |\nabla u|\,d\sigma , \quad \mA_1(r)\defeq\int_{\Xi_r} |\nabla u|^2\,d\sigma,
\quad
\mB_1(r)\defeq \mA_1'(r).
\]
Let
\[
\boldsymbol{n}=\frac{\nabla u}{|\nabla u|}
\]
be the unit normal to \(\Xi_r\). 


\bigskip

\begin{lemma}\label{lem:A0-constant}
    $\mA_0(r)$ is a constant for $r\ge 0$.
\end{lemma}

\begin{proof}
    Differentiating in $r$ gives
    \begin{align*}
        \mA_0'(r)=\int_{\Xi_r} \frac{\Delta u}{|\nabla u|}d\sigma=0,
    \end{align*}
    since $u$ is harmonic.
\end{proof}
Then we derive a differential inequality for $\mA_1(r)$.

\begin{lemma}\label{lem:parabolic-derivative-A}
    We have
\begin{equation}
\mA'_1(r)=\int_{\Xi_r} \nabla^2 u (\boldsymbol{n},\boldsymbol{n})d\sigma=\mB_1(r).
\end{equation}
\end{lemma}

\begin{proof}
 We first rewrite $\mA_1(r)$ as follows. 
\[
\mA_1(r)=\int_{\Xi_r} |\nabla u|^2\,d\sigma=\int_{\Xi_r} \left\langle |\nabla u|\nabla u, \frac{\nabla u}{|\nabla u|} \right\rangle \,d\sigma .
\]
Then the divergence theorem yields
\[
\mA_1(r)-\mA_1(r_0)
=
\int_{\{r_0\le u\le r\}} \divsymb(|\nabla u|\nabla u)\,dV.
\]
Since \(\Delta u=0\), direct calculation yields
\begin{align*}
\divsymb(|\nabla u|\nabla u)&=\langle \nabla |\nabla u|,\nabla u\rangle=
|\nabla u|\nabla^2 u (\boldsymbol{n},\boldsymbol{n}).
\end{align*}
Hence,
\[
\mA_1(r)-\mA_1(r_0)
=
\int_{\{r_0\le u\le r\}} |\nabla u|\,\nabla^2 u (\boldsymbol{n},\boldsymbol{n})\,dV.
\]
Applying the co-area formula yields
\[
\mA(r)-\mA(r_0)
=
\int_{r_0}^r \left(\int_{\Sigma_t} \nabla^2 u (\boldsymbol{n},\boldsymbol{n})\,d\sigma\right)dt.
\]
Differentiating in \(r\) gives
\[
\mA'_1(r)=\int_{\Xi_r} \nabla^2 u (\boldsymbol{n},\boldsymbol{n})\,d\sigma.
\]
\end{proof}

\begin{lemma}
If \(r_1<r_2\) are regular values, then
\begin{equation}\label{eq:formula-of-B_1}
\mB_1(r_2)-\mB_1(r_1)
=
\int_{r_1}^{r_2}\int_{\Xi_r}
\left(
\frac12 \frac{|\nabla^2u|^2}{|\nabla u|^2}
+\frac12 R_g
-\frac12 R_{\Xi_r}
\right)\,d\sigma\,dr
\ge 0.
\end{equation}
\end{lemma}

\begin{proof}
We begin with
\begin{align*}
\mB_1(r)
=
\int_{\Xi_r} \nabla^2 u (\boldsymbol{n},\boldsymbol{n})\,d\sigma
=
\int_{\Xi_r} \left\langle \nabla |\nabla u|, \boldsymbol{n} \right\rangle \,d\sigma.
\end{align*}
By the divergence theorem,
\[
\mB_1(r_2)-\mB_1(r_1)
=
\int_{\{r_1\le u\le r_2\}} \Delta |\nabla u|\,dV.
\]
Applying the coarea formula, we obtain
\begin{equation}\label{eq:derivative-B}
\mB_1(r_2)-\mB_1(r_1)
=
\int_{r_1}^{r_2}
\left(
\int_{\Xi_t}\frac{\Delta |\nabla u|}{|\nabla u|}\,d\sigma
\right)dt.
\end{equation}

Since \(u\) is harmonic, Bochner's formula gives
\[
\frac12 \Delta |\nabla u|^2
=
|\nabla^2u|^2+\operatorname{Ric}(\nabla u,\nabla u).
\]
On the other hand,
\[
\Delta |\nabla u|^2
=
2|\nabla u|\Delta |\nabla u|+2|\nabla |\nabla u||^2.
\]
Therefore,
\[
2|\nabla u|\Delta |\nabla u|+2|\nabla |\nabla u||^2
=
2|\nabla^2u|^2+2\operatorname{Ric}(\nabla u,\nabla u),
\]
and hence
\[
\Delta |\nabla u|
=
\frac{1}{|\nabla u|}
\left(
|\nabla^2u|^2-|\nabla |\nabla u||^2+\operatorname{Ric}(\nabla u,\nabla u)
\right).
\]

Moreover, the Gauss--Codazzi equation yields
\begin{align*}
\Ric (\nabla u, \nabla u)|\nabla u|^{-2}
=
\frac{1}{2}\left(R_g-R_{\Xi_r}\right)
+\frac{1}{|\nabla u|^2}\left(|\nabla |\nabla u||^2-\frac{1}{2}|\nabla^2 u|^2\right).
\end{align*}
Substituting this into the expression for \(\Delta |\nabla u|\), we obtain
\[
\Delta |\nabla u|
=
\frac{1}{2|\nabla u|}
\left( |\nabla^2u|^2+ R_g|\nabla u|^2- R_{\Xi_r}|\nabla u|^2
\right).
\]

Because \(R_g\ge 0\) and \(\mathrm{genus}(\Xi_r)\ge 1\), the Gauss--Bonnet formula implies
\[
\int_{\Xi_r}R_{\Xi_r}\,d\sigma=4\pi\chi(\Xi_r)\le 0,
\]
and therefore
\[
\frac12\int_{\Xi_r}(R_g-R_{\Xi_r})\,d\sigma\ge 0.
\]
Dividing by \(|\nabla u|\) and substituting the resulting expression into \eqref{eq:derivative-B}, we obtain \eqref{eq:formula-of-B_1}. This completes the proof.
\end{proof}

\begin{lemma}\label{lem:parabolic-derivative-a}
For each regular value \(r\), define
\[
\ha(r):=\frac{\mA_1'(r)}{\mA_1(r)}=\bigl(\log \mA_1\bigr)'(r),
\]
and let
\[
V:=\frac{\nabla |\nabla u|}{\mA_1(r)}
\]
be the associated \(C^1\) vector field. Then, for every regular value \(r\),
\begin{equation}\label{eq:parabolic-V-normal}
\int_{\Xi_r}\langle V,\mathbf n\rangle\,d\sigma=\ha(r),
\end{equation}
and
\begin{equation}\label{eq:parabolic-divV-lower}
\int_{\Xi_r}\frac{\operatorname{div}V}{|\nabla u|}\,d\sigma
\ge -\frac12\ha(r)^2.
\end{equation}
\end{lemma}

\begin{proof}
Since \(\mathbf n=\nabla u/|\nabla u|\), we have
\[
\langle \nabla |\nabla u|,\mathbf n\rangle
=\nabla^2u(\mathbf n,\mathbf n).
\]
Hence, by Lemma~\ref{lem:parabolic-derivative-A},
\[
\int_{\Xi_r}\langle V,\mathbf n\rangle\,d\sigma
=
\frac{1}{\mA_1(r)}\int_{\Xi_r}\nabla^2u(\mathbf n,\mathbf n)\,d\sigma
=
\frac{\mA_1'(r)}{\mA_1(r)}
=\ha(r),
\]
which proves \eqref{eq:parabolic-V-normal}.

Next, since \(V=\mA_1(r)^{-1}\nabla |\nabla u|\), we compute
\[
\operatorname{div}V
=
\frac{\Delta |\nabla u|}{\mA_1(r)}
-\frac{\mA_1'(r)}{\mA_1(r)^2}\,\langle \nabla u,\nabla |\nabla u|\rangle.
\]
Dividing by \(|\nabla u|\) and restricting to \(\Xi_r\), we obtain
\[
\frac{\operatorname{div}V}{|\nabla u|}
=
\frac{1}{\mA_1(r)}\,\frac{\Delta |\nabla u|}{|\nabla u|}
-\frac{\ha(r)}{\mA_1(r)}\,\nabla^2u(\mathbf n,\mathbf n).
\]
Since \(u\) is harmonic, Bochner's formula together with the Gauss--Codazzi equation yields
\[
\frac{\Delta |\nabla u|}{|\nabla u|}
=
\frac12\,\frac{|\nabla^2u|^2}{|\nabla u|^2}
+\frac12\,(R_g-R_{\Xi_r}).
\]
Therefore,
\[
\int_{\Xi_r}\frac{\operatorname{div}V}{|\nabla u|}\,d\sigma
=
\frac{1}{2\mA_1(r)}\int_{\Xi_r}\frac{|\nabla^2u|^2}{|\nabla u|^2}\,d\sigma
+\frac{1}{2\mA_1(r)}\int_{\Xi_r}(R_g-R_{\Xi_r})\,d\sigma
-\ha(r)^2.
\]

Because \(R_g\ge 0\) and \(\mathrm{genus}(\Xi_r)\ge 1\), the Gauss--Bonnet formula implies
\[
\int_{\Xi_r}R_{\Xi_r}\,d\sigma=4\pi\chi(\Xi_r)\le 0,
\]
and hence
\[
\frac{1}{2\mA_1(r)}\int_{\Xi_r}(R_g-R_{\Xi_r})\,d\sigma\ge 0.
\]
On the other hand, by the Cauchy--Schwarz inequality,
\[
\mB_1(r)^2
=
\left(\int_{\Xi_r}\nabla^2u(\mathbf n,\mathbf n)\,d\sigma\right)^2
\le
\left(\int_{\Xi_r}|\nabla u|^2\,d\sigma\right)
\left(\int_{\Xi_r}\frac{\bigl(\nabla^2u(\mathbf n,\mathbf n)\bigr)^2}{|\nabla u|^2}\,d\sigma\right).
\]
It follows that
\[
\frac{1}{2\mA_1(r)}\int_{\Xi_r}\frac{|\nabla^2u|^2}{|\nabla u|^2}\,d\sigma
\ge
\frac{1}{2\mA_1(r)}\int_{\Xi_r}\frac{\bigl(\nabla^2u(\mathbf n,\mathbf n)\bigr)^2}{|\nabla u|^2}\,d\sigma
\ge
\frac12\left(\frac{\mB_1(r)}{\mA_1(r)}\right)^2
=
\frac12\ha(r)^2.
\]
Combining the preceding estimates yields \eqref{eq:parabolic-divV-lower}.
\end{proof}

We are now in a position to derive an integral differential inequality for \(\ha\) and, in turn, a lower bound for its growth.

\begin{proposition}
If \(r_1<r_2\) are regular values, then
\begin{equation}\label{eq:parabolic-a-integral}
\ha(r_2)-\ha(r_1)\ge -\frac12\int_{r_1}^{r_2} \ha(r)^2\,dr.
\end{equation}
\end{proposition}

\begin{proof}
By \eqref{eq:parabolic-V-normal}, the divergence theorem, and the coarea formula,
\[
\ha(r_2)-\ha(r_1)
=
\int_{\Xi_{r_2}}\langle V,\mathbf n\rangle\,d\sigma
-\int_{\Xi_{r_1}}\langle V,\mathbf n\rangle\,d\sigma
=
\int_{\{r_1\le u\le r_2\}}\operatorname{div}V\,dV
\]
and hence
\[
\ha(r_2)-\ha(r_1)
=
\int_{r_1}^{r_2}
\left(
\int_{\Xi_r}\frac{\operatorname{div}V}{|\nabla u|}\,d\sigma
\right)\,dr.
\]
Applying \eqref{eq:parabolic-divV-lower} yields \eqref{eq:parabolic-a-integral}.
\end{proof}

\begin{corollary}\label{cor:parabolic-a-lower}
Assume that \(r_1\) is a regular value and that \(\ha(r_1)>0\). Then, for every regular value \(r\ge r_1\),
\begin{equation}\label{eq:parabolic-a-lower}
\ha(r)\ge \frac{2}{\,r-r_1+\frac{2}{\ha(r_1)}\,}.
\end{equation}
\end{corollary}

\begin{proof}
Define
\[
\phi(r):=\frac{2}{\,r-r_1+\frac{2}{\ha(r_1)}\,}.
\]
Then \(\phi(r_1)=\ha(r_1)\) and
\[
\phi'(r)=-\frac12\,\phi(r)^2.
\]
Equivalently,
\[
\phi(r)-\phi(r_1)=-\frac12\int_{r_1}^r \phi(s)^2\,ds.
\]

We claim that \(\ha(r)\ge \phi(r)\) for all regular values \(r\ge r_1\). Suppose otherwise. Then, by continuity of \(\ha\), there exists a first regular value \(r_*>r_1\) such that
\[
\ha(r_*)=\phi(r_*),\qquad \ha(r)<\phi(r)\quad\text{for all regular }r\in(r_1,r_*).
\]
Using \eqref{eq:parabolic-a-integral} together with the inequality \(\ha<\phi\) on \((r_1,r_*)\), we obtain
\[
\ha(r_*)-\ha(r_1)\ge -\frac12\int_{r_1}^{r_*} \ha(s)^2\,ds
> -\frac12\int_{r_1}^{r_*}\phi(s)^2\,ds
=\phi(r_*)-\phi(r_1).
\]
Since \(\ha(r_1)=\phi(r_1)\) and \(\ha(r_*)=\phi(r_*)\), this is a contradiction. Therefore, \eqref{eq:parabolic-a-lower} holds.
\end{proof}

Similar to \cite{Colding-Minicozzi2025}, we now establish a quadratic lower bound for the growth of \(\mA_1(r)\).

\begin{proposition}\label{prop:parabolic-quadratic-growth}
Assume that \(r_1\) is a regular value such that
\[
\mB_1(r_1)=\mA_1'(r_1)>0.
\]
Then there exists a constant \(c>0\) such that
\begin{equation}\label{eq:parabolic-A1-quadratic}
\mA_1(r)\ge c\,r^2
\end{equation}
for all sufficiently large regular values \(r\).
\end{proposition}

\begin{proof}
Since \(\mB_1(r_1)>0\), we have
\[
\ha(r_1)=\frac{\mA_1'(r_1)}{\mA_1(r_1)}>0.
\]
Therefore, by Corollary~\ref{cor:parabolic-a-lower},
\[
\ha(r)\ge \frac{2}{\,r-r_1+\frac{2}{\ha(r_1)}\,}
\]
for every regular value \(r\ge r_1\). Integrating this inequality, we obtain
\[
\log \frac{\mA_1(r)}{\mA_1(r_1)}
=
\int_{r_1}^r \ha(s)\,ds
\ge
\int_{r_1}^r \frac{2}{\,s-r_1+\frac{2}{\ha(r_1)}\,}\,ds
=
2\log\frac{r-r_1+\frac{2}{\ha(r_1)}}{\frac{2}{\ha(r_1)}}.
\]
Exponentiating yields
\[
\mA_1(r)\ge \mA_1(r_1)
\left(
\frac{r-r_1+\frac{2}{\ha(r_1)}}{\frac{2}{\ha(r_1)}}
\right)^2.
\]
Hence \(\mA_1(r)\ge c\,r^2\) for all sufficiently large \(r\), for some constant \(c>0\).
\end{proof}

\begin{remark}
    If $\mA_1(r)$ is twice differentiable, then there is a more direct proof of its quadratic growth through differential inequality. First we have
    \begin{equation}
       \mA_1''(r)=\mB_1'(r)=\int_{\Xi_r}
\frac12\left(
 \frac{|\nabla^2u|^2}{|\nabla u|^2}
+ R_g
- R_{\Xi_r}
\right)\,d\sigma.
    \end{equation}
    In particular, when \(R_g\ge 0\) on $E$, it holds that
\begin{equation}\label{eq:without-R-and-K}
    \mB_1'(r)\ge\frac12\int_{\Xi_r}\frac{|\nabla^2u|^2}{|\nabla u|^2}\,d\sigma\ge 0.
\end{equation}

Recall that
\[
\mA'_1(r)=\mB_1(r)=\int_{\Xi_r} \nabla^2 u (\boldsymbol{n},\boldsymbol{n})d\sigma.
\]
Cauchy--Schwarz gives
\[
\mB_1(r)^2
\le
\left(\int_{\Xi_r}|\nabla u|^2\,d\sigma\right)
\left(\int_{\Xi_r}\frac{\left(\nabla^2 u (\boldsymbol{n},\boldsymbol{n})\right)^2}{|\nabla u|^2}\,d\sigma\right).
\]
Since \(\left(\nabla^2 u (\boldsymbol{n},\boldsymbol{n})\right)^2\le |\nabla^2u|^2\), we have
\[
\mB_1(r)^2
\le
\mA_1(r)\int_{\Xi_r}\frac{|\nabla^2u|^2}{|\nabla u|^2}\,d\sigma.
\]
Using \eqref{eq:without-R-and-K}, we get
\[
\mA_1''(r)=\mB_1'(r)
\ge
\frac12\int_{\Xi_r}\frac{|\nabla^2u|^2}{|\nabla u|^2}\,d\sigma
\ge
\frac{\mB_1(r)^2}{2\mA_1(r)}.
\]
Since \(\mB_1(r)=\mA_1'(r)\), it becomes
\[
\mA_1''(r)\ge \frac{(\mA_1'(r))^2}{2\mA_1(r)}.
\]

Now compute
\[
(\sqrt{\mA_1})''(r)
=
\frac{\mA_1''(r)}{2\sqrt{\mA_1(r)}}-\frac{(\mA_1'(r))^2}{4\mA_1(r)^{3/2}}
=
\frac{2\mA_1(r)\mA_1''(r)-(\mA_1'(r))^2}{4\mA_1(r)^{3/2}}.
\]
Because
\[
2\mA_1(r)\mA_1''(r)-(\mA_1'(r))^2\ge 0,
\]
we conclude
\[
(\sqrt{\mA_1(r)})''\ge 0.
\]
In particular, \(\sqrt{\mA_1(r)}\) is convex. By convexity we have that 
\begin{equation}
    \sqrt{\mA_1}(r)\ge\sqrt{\mA_1}(r_0)+\frac{\mB_1(r_0)}{2\sqrt{\mA_1}(r_1)}(r-r_0),
\end{equation}
So, if $\mB_1(r_0)>0$, there exists some $c>0$ such that for $r\ge r_0$, 
\begin{equation}
    \mA_1(r)\ge cr^2.
\end{equation}
\end{remark}

\begin{lemma}
For all \(r\ge r_0\), we have \(\mB_1(r)\le 0\).
\end{lemma}

\begin{proof}
Suppose instead that there exists \(r_1\ge r_0\) such that
\[
\mB_1(r_1)>0.
\]
Then, by Lemma \ref{prop:parabolic-quadratic-growth}, there exists a constant \(c>0\) such that, for all \(r\ge r_0\),
\begin{equation}\label{eq:A1-at-least-quadratic}
\mA_1(r)\ge cr^2.
\end{equation}

On the other hand, by the Cheng--Yau gradient estimate,
\begin{equation*}
|\nabla u(x)|\le C(n)\left(\frac{1}{d_g(x,\partial E)}+\sqrt{|K|}\right)u(x).
\end{equation*}
Taking the supremum over the compact level set \(\Xi_r=\{u=r\}\), and noting that \(\partial E=\Xi_0\) (here we still regard \(u\) as a function on \(\overline E\)), we see that, as \(r\to\infty\), the distance \(d_g(\Xi_r,\Xi_0)\) is uniformly bounded away from \(0\). Therefore, for all sufficiently large \(r\),
\begin{equation}\label{eq:gradient-linear}
\sup_{x\in \Xi_r}|\nabla u(x)|\le Cr.
\end{equation}
Combining this with Lemma \ref{lem:A0-constant}, we obtain from \eqref{eq:gradient-linear} that
\begin{equation}\label{eq:A1-at-most-linear}
\mA_1(r)= \int_{\Xi_r}|\nabla u|^2 \, d\sigma
\le \sup_{x\in \Xi_r}|\nabla u(x)| \,\mA_0(r)\le Cr.
\end{equation}
For sufficiently large \(r\), \eqref{eq:A1-at-least-quadratic} contradicts \eqref{eq:A1-at-most-linear}. This completes the proof.
\end{proof}

We conclude the subsection by finishing the proof of Proposition \ref{prop:Parabolic-exhaustion}.
\begin{proof}[Proof of Proposition \ref{prop:Parabolic-exhaustion}]
Recall that we have assumed that none of the regular level set $\Sigma_r\cong\mb S^2$ when $r\ge r_0$, so it suffices to show $\Sigma_r\cong\mb T^2$. Following the computation in the previous lemmas, we now have
\[
\mA_1'(r)=\mB_1(r)\le 0 \qquad \text{for all } r\ge r_0.
\]

Since \(\mB_1\) is monotone nondecreasing and bounded above by \(0\), it admits a limit
\[
\ell\defeq\lim_{r\to\infty}\mB_1(r)\le 0.
\]
If \(\ell<0\), then there exists \(r_1\) sufficiently large such that, for all \(r\ge r_1\),
\begin{equation}\label{eq:B_1-strictly-negative}
\mB_1(r)\le \frac{\ell}{2}<0.
\end{equation}
Integrating \eqref{eq:B_1-strictly-negative} from \(r_1\) to \(r>r_1\), we obtain
\[
0\le \mA_1(r)\le \mA_1(r_1)+\frac{\ell}{2}(r-r_1)\to -\infty
\qquad \text{as } r\to\infty,
\]
which is impossible. Therefore,
\[
\lim_{r\to\infty}\mB_1(r)=0.
\]

By \eqref{eq:formula-of-B_1}, we obtain
\begin{align*}
-\mB_1(r)
&=\frac12\int_r^{\infty}\int_{\Xi_t}
\left(
\frac{|\nabla^2u|^2}{|\nabla u|^2}+R_g-R_{\Xi_t}
\right)\,d\sigma\,dt\\
&=\frac12\int_r^{\infty}
\left(
\int_{\Xi_t}
\left(
\frac{|\nabla^2u|^2}{|\nabla u|^2}+R_g
\right)\,d\sigma
+4\pi(\mathrm{genus}(\Xi_t)-1)
\right)\,dt.
\end{align*}
Since \(\mB_1(r)\to 0\), we may choose \(r\) sufficiently large so that
\[
-\mB_1(r)<1.
\]
It follows that
\[
\frac12\int_r^{\infty}
\left(
\int_{\Xi_t}
\left(
\frac{|\nabla^2u|^2}{|\nabla u|^2}+R_g
\right)\,d\sigma
+4\pi(\mathrm{genus}(\Xi_t)-1)
\right)\,dt<1.
\]
Since \(\mathrm{genus}(\Xi_t)\ge 1\) for every regular \(t\), the integrand is nonnegative. In particular, we must have
\[
\mathrm{genus}(\Xi_t)=1
\]
for all sufficiently large regular values \(t\). Hence every such \(\Xi_t\) is a torus. Therefore, \(M\) admits an exhaustion with torus boundaries, as desired.
\end{proof}

\subsection{Nonparabolic case}
In this case, rather than considering a harmonic function on the end, we work with the minimal Green's function $G$ on $M\setminus\{p\}$, where $p\in M$ is the pole. The function $G$ is positive and harmonic on $M\setminus\{p\}$. Motivated by the behavior of the Green's function on Euclidean space $\R^3$, we define the \emph{Green distance function} by
\[
b \defeq \frac{1}{G}.
\]
Then $b$ is positive and satisfies
\[
\Delta b^2 = 6|\nabla b|^2.
\]
We also define the trace-free part of the hessian of $b^2$ as
\begin{align*}
    \boldsymbol{\mathrm{B}}\defeq \nabla^2 b^2-2|\nabla b|^2g.
\end{align*}

\begin{proposition}\label{prop:Nonparabolic-exhaustion}
    Let $(M,g)$ be a complete Riemannian $3$-manifold with one end, finite first Betti number and $R_g\ge 0$. If $(M,g)$ is nonparabolic and satisfies the regularity assumptions \eqref{eq:Ricci-lower-bound} and \eqref{eq:Green-vanish}, then $M$ admits an exhaustion $\Omega_i$ such that for all $i$, we have  $\partial\Omega_i\cong \mb S^2$.
\end{proposition}
It is easy to see that, whenever \eqref{eq:Green-vanish} holds, the positive level sets of $G$, and hence those of $b$, are compact. Equivalently, both $G$ and $b$ are proper functions on $M\setminus\{p\}$.

We now refine several main results from the work of Colding--Minicozzi \cite{Colding-Minicozzi2025}. Under our assumptions, there exists $r_0>0$ such that, for every regular value $r\ge r_0$, the level set
\[
\Sigma_r:=\{b=r\}
\]
is a smooth, closed, connected, oriented surface of genus at least one.

We begin by recalling some local quantities:
\begin{align*}
A_1(r) &= r^{-2}\int_{ \{b=r\} } |\nabla b|^2,\quad B_1(r) = r^{-2}\int_{\{b=r\}} \nabla^2 b^2(\boldsymbol{n},\boldsymbol{n}),\\
B_2(r) &= \int_{\{b=r\}} \frac{|\boldsymbol{\mathrm{B}}|^2}{|\nabla b^2|^2}, \quad S_1(r) = \int_{\{b=r\}} R_g.
\end{align*}

\begin{lemma}
The function $A_1$ is continuously differentiable and satisfies
\begin{equation}
rA_1'(r)=\frac{1}{2}B_1(r)-A_1(r)
=\frac{1}{2}r^{-2}\int_{b=r}\boldsymbol{\mathrm{B}}(\boldsymbol{n},\boldsymbol{n}),
\end{equation}
as well as
\begin{equation}
2\bigl(rA_1(r)\bigr)'=B_1(r).
\end{equation}
\end{lemma}

The following proposition follows from a slight modification of the proof of \cite{Colding-Minicozzi2025}*{Proposition~1.30}, together with our assumption that each level set $\Sigma_r$ has genus at least one.

\begin{proposition}[\cite{Colding-Minicozzi2025}*{Proposition~1.30}]\label{prop:first-derivative-A}
For every regular value $r>r_0$, we have
\begin{equation}
rB_1(r)\ge 4rA_1(r)+\int_{r_0}^r \bigl(S_1(s)+B_2(s)\bigr)\,ds.
\end{equation}
Equivalently,
\begin{equation}
rA_1'(r)\ge A_1(r)+\frac{1}{2r}\int_{r_0}^r \bigl(S_1(s)+B_2(s)\bigr)\,ds.
\end{equation}
\end{proposition}

Next, define the continuous positive function
\[
a(r)=r(\log A_1)'(r),
\]
which measures the rate of polynomial growth of $A_1(r)$. To show that $A_1$ grows quadratically, we will prove that $a(r)$ converges rapidly to $2$ as $r\to\infty$.

A direct computation, together with our assumption that the level set $\Sigma_r$ has genus at least one, yields the following lemma.

\begin{lemma}[\cite{Colding-Minicozzi2025}*{Lemma 2.15}]
Define a $C^1$ vector field by
\begin{equation}
V=\frac{b^{-2}\nabla |\nabla b^2|}{A_1(r)}.
\end{equation}
Then, at each regular value $r$, we have
\begin{equation}
\int_{b=r} \langle V,\boldsymbol{n}\rangle = 2a(r)+2,
\end{equation}
and
\begin{equation}
r\int_{b=r} \frac{\operatorname{div}V}{|\nabla b|}
\ge 2-\frac{1}{2}a(r)^2.
\end{equation}
\end{lemma}

The next proposition shows that the continuous function $a(r)$ satisfies the differential inequality
\[
ra'(r)\ge 1-\frac{a(r)^2}{4}
\]
in an integral sense. Its proof follows from the preceding lemma and the coarea formula.

\begin{proposition}[\cite{Colding-Minicozzi2025}*{Proposition 2.13}]
If $r_1<r_2$ are regular values of $b$ with $r_0\le r_1$, then
\begin{equation}
a(r_2)-a(r_1)\ge \int_{r_1}^{r_2}
\left(1-\frac{a(r)^2}{4}\right)\frac{dr}{r}.
\end{equation}
\end{proposition}

\begin{corollary}[\cite{Colding-Minicozzi2025}*{Corollary 2.27}]\label{coro:monotone-a}
    If $r_1<r_2$ are regular values of $b$ with $r_0\le r_1$ and $a(r)\le 2$ for
$r_1\le r\le r_2$, then
\begin{equation}
a(r_2)\ge
2+\sqrt{\frac{r_1}{r_2}}\bigl(a(r_1)-2\bigr).
\end{equation}
\end{corollary}

We are now in a position to establish a quadratic lower bound for the growth of $A_1(r)$.
\begin{proposition}[\cite{Colding-Minicozzi2025}*{Corollary 2.33, proof pf Proposition 0.4}]\label{prop:quadric-growth-A}
Suppose that $M^3$ is a contractible complete Riemannian $3$-manifold with $R_g\ge 0$. If $A_1(r_0)>0$ for some $r_0>0$, then there exists a constant $c>0$ such that
\begin{equation}
A_1(r)\ge c r^2
\end{equation}
for all $r\ge r_0$.
\end{proposition}

\begin{proof}
We first claim that there exists a constant $c>0$ such that, for all $r\ge r_0$,
\begin{equation}\label{eq:lower-bound-a}
a(r)\ge 2-\frac{c}{\sqrt r}.
\end{equation}
Since $a(r)$ is continuous and the regular values of $b$ are dense, it suffices to verify \eqref{eq:lower-bound-a} for regular values of $b$. By Corollary \ref{coro:monotone-a}, we have
\begin{equation}
a(r)\ge 2+\sqrt{\frac{r_0}{r}}\bigl(a(r_0)-2\bigr).
\end{equation}
If $a(r_0)<2$, we choose $c=(2-a(r_0))\sqrt{r_0}$, otherwise, we choose $c=0$.

Since $M$ has one end and first Betti number zero, the level sets of $b$ are connected, and each complement has exactly one bounded component and one unbounded component, along which $b\to\infty$. Therefore, all of the results established in this subsection apply.

By the claim just proved, there exist constants $c>0$ and $r_0>0$ such that, for all $r\ge r_0$,
\begin{equation}
r\frac{A_1'(r)}{A_1(r)}=a(r)\ge 2-\frac{c}{\sqrt r}.
\end{equation}
Integrating from $r_0$ to $r$, we obtain
\begin{equation}
\log\frac{A_1(r)}{A_1(r_0)}
=\int_{r_0}^r \frac{a(s)}{s}\,ds
\ge \int_{r_0}^r \left(2s^{-1}-cs^{-3/2}\right)\,ds
\ge 2\log\frac{r}{r_0}-2c r_0^{-1/2}.
\end{equation}
Exponentiating yields
\begin{equation}
\frac{A_1(r)}{A_1(r_0)}
\ge \frac{r^2}{r_0^2} e^{-2c r_0^{-1/2}}.
\end{equation}
This proves the desired quadratic lower bound for $A_1(r)$.
\end{proof}

\begin{proof}[Proof of  Proposition \ref{prop:Nonparabolic-exhaustion}]
We argue by contradiction. Suppose that there exists $r_0>0$ such that, for every regular value $r\ge r_0$, the level set
\[
\Sigma_r:=\{b=r\}
\]
is a smooth, closed, connected, orientable surface of genus at least one. If there exists $\delta>0$ with $A_1(r_0)\ge \delta$, then Proposition \ref{prop:quadric-growth-A} implies that, for all $r\ge r_0$,
\[
A_1(r)\ge c r^2,
\]
where $c=c(\delta,r_0)>0$.

On the other hand, in view of the assumption \eqref{eq:Green-vanish} that $G(x)\to 0$ as $x\to\infty$, the Cheng--Yau gradient estimate yields
\[
|\nabla G|(x)\to 0,\text{ as } x\to\infty.
\]
Consequently, we have
\begin{align*}
r^{-2}A_1(r)
&=r^{-2}\int_{\{b=r\} } |\nabla b|^2\, d\sigma \\
&=r^{-2}\int_{\{b=r\} } |\nabla G|\,|\nabla b|\, d\sigma \\
&\le  \, r^{-2}\int_{\{b=r\} } |\nabla b|\, d\sigma \sup_{\{b=r\}} |\nabla G| \\
&\to 0
\end{align*}
as $r\to\infty$, a contradiction.

This would force $A_1(r)=0$ for $r\ge r_0$, and hence $\nabla b=\nabla u=0$ in $\{b\ge r_0\}$, contradicting the fact that $u$ is a positive Green's function. Therefore, $M$ admits an exhaustion whose boundaries are spheres, as desired.
\end{proof}

\bibliographystyle{amsalpha} 
\bibliography{contractible}

\end{document}